\documentclass{article}
\usepackage{amsfonts}
\usepackage{amssymb}
\usepackage{amsmath}

\newtheorem{theorem}{Theorem}

\newtheorem{corollary}[theorem]{Corollary}

\newtheorem{lemma}[theorem]{Lemma}

\newtheorem{proposition}[theorem]{Proposition}

\newenvironment{proof}[1][Proof]{\noindent\textbf{#1.} }{\ \rule{0.5em}{0.5em}}

\begin{document}

\title{{\LARGE \textbf{Blow up of mild solutions of a system of partial
differential equations with distinct fractional diffusions}}}
\author{ Jos\'{e} Villa-Morales\thanks{
Supported by the grants No.118294 of CONACyT and PIM10-2N of UAA.}}
\date{ }
\maketitle

\begin{abstract}
We give a sufficient condition for blow up of positive mild solutions to an
initial value problem for a nonautonomous weakly coupled system with
distinct fractional diffusions. The proof is based on the study of blow up
of a particular system of ordinary differential equations.
\end{abstract}



\section{Introduction: statement of the result and ov\-er\-view}

Let $i\in \{1,2\}$ and $j=3-i$. In this paper we study blow up of positive
mild solutions of 
\begin{eqnarray}
\frac{\partial u_{i}\left( t,x\right) }{\partial t} &=&g_{i}\left( t\right)
\Delta _{\alpha _{i}}u_{i}\left( t,x\right) +h_{i}\left( t\right)
u_{j}^{\beta _{i}}\left( t,x\right) ,\text{\ \ }t>0,\text{ }x\in \mathbb{R}%
^{d},  \label{wcs} \\
u_{i}\left( 0,x\right) &=&\varphi _{i}\left( x\right) ,\text{\ \ }x\in 
\mathbb{R}^{d},  \notag
\end{eqnarray}%
where $\Delta _{\alpha _{i}}=-\left( -\Delta \right) ^{\alpha _{i}/2}$, $%
0<\alpha _{i}\leq 2$, is the $\alpha _{i}$-Laplacian, $\beta _{i}\geq 1$ are
constants, $\varphi _{i}$ are non negative, not identically zero, bounded
continuous functions and $h_{i},g_{i}:(0,\infty )\rightarrow \left[ 0,\infty
\right) $ are continuous functions.

If there exist a solution $\left( u_{1},u_{2}\right) $ of (\ref{wcs})
defined in $\left[ 0,\infty \right) \times \mathbb{R}^{d}$, we say that $%
\left( u_{1},u_{2}\right) $ is a global solution, on the other hand if there
exists a number $t_{e}<\infty $ such that $\left( u_{1},u_{2}\right) $ is
unbounded in $\left[ 0,t\right] \times \mathbb{R}^{d}$, for each $t>t_{e}$,
we say that $\left( u_{1},u_{2}\right) $ blows up in finite time.

The associated integral system of (\ref{wcs}) is%
\begin{eqnarray}
u_{i}(t,x) &=&\int_{\mathbb{R}^{d}}p_{i}\left( G_{i}\left( t\right)
,y-x\right) \varphi _{i}(y)dy  \label{ecintegrl} \\
&&+\int\nolimits_{0}^{t}\int\nolimits_{\mathbb{R}^{d}}p_{i}\left(
G_{i}\left( s,t\right) ,y-x\right) h_{i}\left( s\right) u_{j}^{\beta
_{i}}\left( s,y\right) dyds.  \notag
\end{eqnarray}%
Here $p_{i}\left( t,x\right) $ denote the fundamental solution of $\frac{%
\partial }{\partial t}-\Delta _{\alpha _{i}}$ and 
\begin{equation*}
G_{i}\left( s,t\right) =\int\nolimits_{s}^{t}g_{i}\left( r\right) dr,\ \
0\leq s\leq t,
\end{equation*}%
where $G_{i}(t)=G_{i}(0,t)$. We say that $\left( u_{1},u_{2}\right) $ is a
mild solution of (\ref{wcs}) if $\left( u_{1},u_{2}\right) $ is a solution
of (\ref{ecintegrl}).

The main result is:

\begin{theorem}
\label{TeoPr}Assume that $\beta _{i}\beta _{j}>1$ and%
\begin{equation}
\lim_{t\rightarrow \infty }G_{i}\left( t\right) =\infty .  \label{condG}
\end{equation}%
Let $a\in \{1,2\}$ such that 
\begin{equation}
\alpha _{a}=\min \{\alpha _{1},\alpha _{2}\}\text{ \ and \ }b=3-a.
\label{defa}
\end{equation}%
Define%
\begin{equation}
f_{i}(t)=h_{i}\left( t\right) \left( \frac{G_{b}\left( t\right) }{%
(G_{j}\left( t\right) ^{\alpha _{b}/\alpha _{j}}+G_{b}\left( t\right)
)^{\beta _{i}}}\right) ^{d/\alpha _{b}},\ \ t>0.  \label{dfi}
\end{equation}%
Then the positive solution of (\ref{ecintegrl}) blows up in finite time if 
\begin{equation}
\int_{\cdot }^{\infty }F(s)ds=\infty ,  \label{cenfexplo}
\end{equation}%
where%
\begin{equation}
F(t)=\left( f_{i}(t)^{1/(\beta _{i}+1)}f_{j}(t)^{1/(\beta _{j}+1)}\right)
^{(\beta _{i}+1)(\beta _{j}+1)/(\beta _{i}+\beta _{j}+2)}.  \label{dfexpl}
\end{equation}
\end{theorem}

It is well known that a classical solution is a mild solution. Therefore, if
we give a sufficient condition for blow up of positive solutions to (\ref%
{ecintegrl}) then we have a condition for blow up of classical solutions to (%
\ref{wcs}).

\begin{corollary}
\label{SegRe}Moreover, assume that $\rho _{i}>0$, $\sigma _{i}>-1$ and 
\begin{eqnarray}
\lefteqn{\tfrac{d\rho _{b}}{\alpha _{b}}+\tfrac{\sigma _{i}(1+\beta
_{j})+\sigma _{j}(1+\beta _{i})}{\beta _{i}+\beta _{j}+2}+1\geq }  \notag \\
&&\tfrac{d}{\beta _{i}+\beta _{j}+2}\left[ \beta _{i}(\beta _{j}+1)\max
\left\{ \tfrac{\rho _{j}}{\alpha _{j}},\tfrac{\rho _{b}}{\alpha _{b}}%
\right\} +\beta _{j}(\beta _{i}+1)\max \left\{ \tfrac{\rho _{i}}{\alpha _{i}}%
,\tfrac{\rho _{b}}{\alpha _{b}}\right\} \right] ,  \label{conexpledi}
\end{eqnarray}%
then each (classical) solution to%
\begin{eqnarray}
\frac{\partial u_{i}\left( t,x\right) }{\partial t} &=&\rho _{i}t^{\rho
_{i}-1}\Delta _{\alpha _{i}}u_{i}\left( t,x\right) +t^{\sigma
_{i}}u_{j}^{\beta _{i}}\left( t,x\right) ,\text{\ \ }t>0,\text{ }x\in 
\mathbb{R}^{d},  \label{Edparpart} \\
u_{i}\left( 0,x\right) &=&\varphi _{i}\left( x\right) ,\text{\ \ }x\in 
\mathbb{R}^{d}.  \notag
\end{eqnarray}%
blow up in finite time.
\end{corollary}

In applied mathematics it is well known the importance of the study of
equations like (\ref{wcs}). In fact, for example, they arise in fields like
molecular biology, hydrodynamics and statistical physics \cite{S-Z}. Also,
notice that generators of the form $g_{i}\left( t\right) \Delta _{\alpha
_{i}}$ arise in models of anomalous growth of certain fractal interfaces 
\cite{M-W}.

There are many related works. Here are some of them:

\begin{itemize}
\item When $\alpha _{1}=\alpha _{2}=2,$ $\rho _{1}=\rho _{2}=1,$ $\sigma
_{1}=\sigma _{2}=0$ and $\varphi _{1}=\varphi _{2}$ in (\ref{Edparpart}),
Fujita \cite{Fujita} showed that if $d<\alpha _{1}/\beta _{1}$, then for any
non-vanishing initial condition the solution of (\ref{Edparpart}) is
infinite for all $t$ large enough.

\item When $\alpha _{1}=\alpha _{2},$ $\rho _{1}=\rho _{2},$ $\sigma
_{1}=\sigma _{2}$ and $\varphi _{1}=\varphi _{2}$ in (\ref{Edparpart}), P%
\'{e}rez and Villa \cite{P-V} showed that if $\sigma _{1}+1\geq d\rho
_{1}(\beta _{1}-1)/\alpha _{1},$ then the solutions of (\ref{Edparpart})
blow up in finite time.

\item When $\alpha _{1}=\alpha _{2}=2$ and $\rho _{1}=\rho _{2}=1$ in (\ref%
{Edparpart}), Uda \cite{Uda} proved that all positive solutions of (\ref%
{Edparpart}) blow up if $\max \left\{ \frac{\left( \sigma _{2}+1\right)
\beta _{1}+\sigma _{1}+1}{\beta _{1}\beta _{2}-1},\frac{\left( \sigma
_{1}+1\right) \beta _{2}+\sigma _{2}+1}{\beta _{1}\beta _{2}-1}\right\} \geq 
\frac{d}{2}$.

\item When $\alpha _{1}=\alpha _{2}$, $g_{1}(t)=g_{2}(t)=t^{\rho -1}$, $\rho
>0$, and $h_{1}(t)=h_{2}(t)=1$ in (\ref{wcs}), P\'{e}rez \cite{P} proved
that every positive solution blows up in finite time if $\min \left\{ \frac{%
\alpha _{1}}{\rho \left( \beta _{1}-1\right) },\frac{\alpha _{1}}{\rho
\left( \beta _{2}-1\right) }\right\} >d$.

\item When $\rho _{1}=\rho _{2}=1$ and the nonlinear terms in (\ref%
{Edparpart}) are of the form $h\left( t,x\right) u^{\beta _{i}}$, $h\left(
t,x\right) =O\left( t^{\sigma }\left\vert x\right\vert ^{\gamma }\right) $,
Guedda and Kirane \cite{G-K-2} also studied blow up.
\end{itemize}

Other related results (when $\alpha _{1}=\alpha _{2}=2$) can be found, for
example in \cite{andre}, \cite{fila}, \cite{koby}, \cite{moch} and
references therein.

It is worth while to mention that Guedda and Kirane \cite{G-K-2} observed
that to reduce the study of blow up of (\ref{wcs}) to a system of ordinary
differential equations we must have a comparison result between $p_{i}\left(
t,x\right) $ and $p_{j}\left( t,x\right) $. Therefore, the goal of this
paper is to use the comparison result given in \cite{M-V} (Lemma 2.4) to
follows the usual approach, see among others \cite{Sug} or \cite{G-K-1}.

When $\alpha _{1}=\alpha _{2}=2$, $\rho _{1}=\rho _{2}=1$ and $\sigma
_{1}=\sigma _{2}=0$ the Uda condition (\ref{dUda}), the P\'{e}rez condition (%
\ref{dAro}) and the condition (\ref{conexpledi}) become%
\begin{eqnarray}
d &\leq &\frac{2(\max \{\beta _{1},\beta _{2}\}+1)}{\beta _{1}\beta _{2}-1}%
=C_{U},  \label{dUda} \\
d &<&\frac{2}{\max \{\beta _{1},\beta _{2}\}-1}=C_{A},  \label{dAro} \\
d &\leq &\frac{\beta _{1}+\beta _{2}+2}{\beta _{1}\beta _{2}-1}=C_{V},
\label{dVilla}
\end{eqnarray}%
respectively. Since $C_{A}\leq C_{V}\leq C_{U}$ we see that the Uda
condition (\ref{dUda}) is the best. Also, from this we see that $C_{V}$,
given in (\ref{dVilla}), is not the optimal bound (critical dimension), but
we believe that it is the best we can get by constructing a convenient
subsolution of the solution of (\ref{ecintegrl}). In fact, the condition (%
\ref{conexpledi}) coincides with the condition for blow up given by P\'{e}%
rez and Villa \cite{P-V}.

The paper is organized as follows. In Section 1 we prove the existence of
local solutions for the equation (\ref{ecintegrl}). In Section 2 we give
some preliminary results and discuses a sufficient condition for blow up of
a system of ordinary differential equations, finally in Section 3 we prove
the main result and its corollary.

\section{Local existence}

The existence of local solutions for the weakly coupled system (\ref%
{ecintegrl}) follows form the fix-point theorem of Banach. We begin
introducing some normed linear spaces. By $L^{\infty }\left( \mathbb{R}%
^{d}\right) $ we denote the space of all real-valued functions essentially
bounded defined on $\mathbb{R}^{d}$. Let $\tau >0$ be a real number that we
will fix later. Define%
\begin{equation*}
E_{\tau }=\left\{ \left( u_{1},u_{2}\right) :\left[ 0,\tau \right]
\rightarrow L^{\infty }\left( \mathbb{R}^{d}\right) \times L^{\infty }\left( 
\mathbb{R}^{d}\right) ,\text{ }|||\left( u_{1},u_{2}\right) |||<\infty
\right\} \text{,}
\end{equation*}%
where%
\begin{equation*}
|||\left( u_{1},u_{2}\right) |||=\sup_{0\leq t\leq \tau }\left\{ \left\Vert
u_{1}\left( t\right) \right\Vert _{\infty }+\left\Vert u_{2}\left( t\right)
\right\Vert _{\infty }\right\} .
\end{equation*}%
Then $E_{\tau }$ is a Banach space and the sets, $R>0$,
\begin{eqnarray*}
P_{\tau } &=&\left\{ \left( u_{1},u_{2}\right) \in E_{\tau }\text{, }%
u_{1}\geq 0,u_{2}\geq 0\right\} , \\
B_{\tau } &=&\left\{ \left( u_{1},u_{2}\right) \in E_{\tau }\text{, }%
|||\left( u_{1},u_{2}\right) |||\leq R\right\} \text{,}
\end{eqnarray*}%
are closed subspaces of $E_{\tau }$.

\begin{theorem}
\label{exisloc}There exists a $\tau =\tau \left( \varphi _{1},\varphi
_{2}\right) >0$ such that the integral system (\ref{ecintegrl}) has a local
solution in $B_{\tau }\cap P_{\tau }$.
\end{theorem}

\begin{proof}
Define the operator $\Psi :B_{\tau }\cap P_{\tau }\rightarrow B_{\tau }\cap
P_{\tau }$, by%
\begin{align*}
\lefteqn{\Psi \left( u_{1},u_{2}\right) \left( t,x\right) } \\
& =\left( \int\nolimits_{\mathbb{R}^{d}}p_{1}\left( G_{1}\left( t\right)
,y-x\right) \varphi _{1}\left( y\right) dy,\int\nolimits_{\mathbb{R}%
^{d}}p_{2}\left( G_{2}\left( t\right) ,y-x\right) \varphi _{2}\left(
y\right) dy\right)  \\
& +\left( \int\nolimits_{0}^{t}\int\nolimits_{\mathbb{R}^{d}}p_{1}\left(
G_{1}\left( s,t\right) ,y-x\right) h_{1}\left( s\right) u_{2}^{\beta
_{1}}\left( s,y\right) dyds,\right.  \\
& \left. \int\nolimits_{0}^{t}\int\nolimits_{\mathbb{R}^{d}}p_{2}\left(
G_{2}\left( s,t\right) ,y-x\right) h_{2}\left( s\right) u_{1}^{\beta
_{2}}\left( s,y\right) dyds\right) .
\end{align*}
We choose $R$ sufficiently large such that $\Psi$ is onto $B_{\tau }\cap
P_{\tau }$. We are going to show that $\Psi $ is a contraction, therefore $\Psi $ has a
fix point. Let $\left( u_{1},u_{2}\right) ,\left( \tilde{u}_{1},\tilde{u}%
_{2}\right) \in B_{\tau }\cap P_{\tau }$ with $u_{i}(0)=\widetilde{u}_{i}(0)$%
,%
\begin{eqnarray*}
&&|||\Psi \left( u_{1},u_{2}\right) -\Psi \left( \tilde{u}_{1},\tilde{u}%
_{2}\right) ||| \\
&=&|||\left( \int\nolimits_{0}^{t}\int\nolimits_{\mathbb{R}^{d}}p_{1}\left(
G_{1}\left( s,t\right) ,y-x\right) h_{1}\left( s\right) \left[ u_{2}^{\beta
_{1}}\left( s,y\right) -\tilde{u}_{2}^{\beta _{1}}\left( s,y\right) \right]
dyds,\right.  \\
&&\left. \int\nolimits_{0}^{t}\int\nolimits_{\mathbb{R}^{d}}p_{2}\left(
G_{2}\left( s,t\right) ,y-x\right) h_{2}\left( s\right) \left[ u_{1}^{\beta
_{2}}\left( s,y\right) -\tilde{u}_{1}^{\beta _{2}}\left( s,y\right) \right]
dyds\right) ||| \\
&\leq &\sum_{i=1}^{2}\sup_{t\in \left[ 0,\tau \right] }\int%
\nolimits_{0}^{t}\int\nolimits_{\mathbb{R}^{d}}p_{i}\left( G_{i}\left(
s,t\right) ,y-x\right) h_{i}\left( s\right) \left\Vert u_{j}^{\beta
_{i}}\left( s\right) -\tilde{u}_{j}^{\beta _{i}}\left( s\right) \right\Vert
_{\infty }dyds.
\end{eqnarray*}%
Let $w,z>0$ and $p\geq 1$, then%
\begin{equation*}
\left\vert w^{p}-z^{p}\right\vert \leq p\left( w\vee z\right)
^{p-1}\left\vert w-z\right\vert .
\end{equation*}%
Using the previous elementary inequality we get
\begin{eqnarray*}
\left\vert u_{j}^{\beta _{i}}\left( s,x\right) -\tilde{u}_{j}^{\beta
_{i}}\left( s,x\right) \right\vert  &\leq &\beta _{i}\left( u_{j}\left(
s,x\right) \vee \tilde{u}_{j}\left( s,x\right) \right) ^{\beta
_{i}-1}\left\vert u_{j}\left( s,x\right) -\tilde{u}_{j}\left( s,x\right)
\right\vert  \\
&\leq &\beta _{i}R^{\beta _{i}-1}\left\Vert u_{j}-\tilde{u}_{j}\right\Vert _{\infty }\text{,}
\end{eqnarray*}%
from this we deduce%
\begin{eqnarray*}
|||\Psi \left( u_{1},u_{2}\right) -\Psi \left( \tilde{u}_{1},\tilde{u}%
_{2}\right) ||| &\leq &\sum_{i=1}^{2}\sup_{t\in \left[ 0,\tau \right]
}\int\nolimits_{0}^{t}h_{i}\left( s\right) \beta _{i}R^{\beta _{i}-1}\left\Vert u_{i}(s)-%
\tilde{u}_{i}\left( s\right) \right\Vert _{\infty }ds \\
&\leq &(\sum_{i=1}^{2}\beta _{i}R^{\beta _{i}-1}\int\nolimits_{0}^{\tau }h_{i}\left(
s\right) ds)|||\left( u_{1},u_{2}\right) -\left( \tilde{u}_{1},\tilde{u}%
_{2}\right) |||.
\end{eqnarray*}%
Since $\lim_{t\rightarrow 0}\int\nolimits_{0}^{t}h_{i}\left( s\right) ds=0$%
, we can choose $\tau >0$ small enough such that $\Psi $ is a
contraction.\hfill 
\end{proof}

\section{Preliminary results}

We begin with:

\begin{lemma}
\label{pden}For any $s,t>0$ and any $x,y\in \mathbb{R}^{d},$ we have\newline
(i) $p_{i}\left( ts,x\right) =t^{-d/\alpha_{i} }p_{i}\left( s,t^{-1/\alpha_{i}}x\right) .$\newline
(ii) $p_{i}\left( t,x\right) \geq \left( \frac{s}{t}\right) ^{d/\alpha_{i}}p_{i}\left( s,x\right) $, for $t\geq s.$\newline
(iii) $p_{i}\left( t,\frac{1}{\tau }\left( x-y\right) \right) \geq
p_{i}\left( t,x\right) p_{i}\left( t,y\right) ,$ if $p_{i}\left( t,0\right)
\leq 1$ and $\tau \geq 2$.\newline
(iv) There exist constants $c_{i}\in \left( 0,1\right] $ such that%
\begin{equation}
p_{i}\left( t,x\right) \geq c_{i}p_{b}(t^{\alpha _{b}/\alpha _{i}},x),
\label{reldensi}
\end{equation}%
where $b$ is as in (\ref{defa}).
\end{lemma}

\begin{proof}
For (i)-(iii) see Section 2 in \cite{Sug} and for (iv) see Lemma 2.4 in \cite%
{M-V}.\hfill
\end{proof}

\begin{lemma}
\label{cotini}Let $u_{i}$ be a positive solution of (\ref{ecintegrl}), then%
\begin{equation}
u_{i}\left( t_{0},x\right) \geq c_{i}(t_{0})p_{b}\left( 2^{-\alpha
_{b}}G_{i}\left( t_{0}\right) ^{\alpha _{b}/\alpha _{i}},x\right) ,\text{ \ }%
\forall x\in \mathbb{R}^{d},  \label{estuini}
\end{equation}%
where%
\begin{equation*}
c_{i}(t_{0})=c_{i}2^{-d}\int\nolimits_{\mathbb{R}^{d}}p_{b}\left(
G_{i}\left( t_{0}\right) ^{\alpha _{b}/\alpha _{i}},2y\right) \varphi
_{i}(y)dy
\end{equation*}%
and $t_{0}>1$ is large enough such that%
\begin{equation}
p_{b}\left( G_{i}\left( t_{0}\right) ^{\alpha _{b}/\alpha _{i}},0\right)
\leq 1.  \label{tre}
\end{equation}
\end{lemma}

\begin{proof}
By (i) of Lemma \ref{pden} and (\ref{condG}) there exist $t_{0}$ large
enough such that 
\begin{equation}
p_{b}\left( G_{i}\left( t_{0}\right) ^{\alpha _{b}/\alpha _{i}},0\right)
=G_{i}\left( t_{0}\right) ^{-d/\alpha _{i}}p_{b}\left( 1,0\right) \leq 1.
\label{cpsm1}
\end{equation}%
Using (iii) and (i) of Lemma \ref{pden}, we get%
\begin{eqnarray*}
p_{b}\left( G_{i}\left( t_{0}\right) ^{\alpha _{b}/\alpha _{i}},y-x\right) 
&\geq &p_{b}\left( G_{i}\left( t_{0}\right) ^{\alpha _{b}/\alpha
_{i}},2x\right) p_{b}\left( G_{i}\left( t_{0}\right) ^{\alpha _{b}/\alpha
_{i}},2y\right)  \\
&=&2^{-d}p_{b}\left( 2^{-\alpha _{b}}G_{i}\left( t_{0}\right) ^{\alpha
_{b}/\alpha _{i}},x\right) p_{b}\left( G_{i}\left( t_{0}\right) ^{\alpha
_{b}/\alpha _{i}},2y\right) \text{.}
\end{eqnarray*}%
From (\ref{ecintegrl}), (iv) of Lemma \ref{pden} and the previous inequality
we conclude%
\begin{equation*}
u_{i}(t_{0},x)\geq \left( c_{i}2^{-d}\int\nolimits_{\mathbb{R}%
^{d}}p_{b}\left( G_{i}\left( t_{0}\right) ^{\alpha _{b}/\alpha
_{i}},2y\right) \varphi _{i}(y)dy\right) p_{b}\left( 2^{-\alpha
_{b}}G_{i}\left( t_{0}\right) ^{\alpha _{b}/\alpha _{i}},x\right) .
\end{equation*}%
Getting the desired result.\hfill 
\end{proof}

Observe that the semigroup property implies 
\begin{eqnarray}
\lefteqn{u_{i}(t+t_{0},x)=\int\nolimits_{\mathbb{R}^{d}}p_{i}\left(
G_{i}\left( t_{0},t+t_{0}\right) ,y-x\right) u_{i}(t_{0},y)dy}  \notag \\
&&+\int\nolimits_{0}^{t}\int\nolimits_{\mathbb{R}^{d}}p_{i}\left(
G_{i}\left( s+t_{0},t+t_{0}\right) ,y-x\right) h_{i}\left( s+t_{0}\right)
u_{j}^{\beta _{i}}\left( s+t_{0},y\right) dyds.  \label{umtei}
\end{eqnarray}%
\hfill

Let 
\begin{equation}
\bar{u}_{i}\left( t\right) =\int\nolimits_{\mathbb{R}^{d}}p_{b}\left(
G_{b}(t),x\right) u_{i}\left( t,x\right) dx,\text{ \ }t\geq 0.  \label{dubar}
\end{equation}

\begin{lemma}
\label{caresp}If $\overline{u}_{i}$ blow up in finite time, then $u_{i}$
also does.
\end{lemma}

\begin{proof}
Let $t_{0}$ be given in Lemma \ref{pden}. Take $t_{0}<t_{j}<\infty $ the
explosion time of $\overline{u}_{j}$. From (\ref{condG}) we can choose $%
t>t_{j}$ large enough such that 
\begin{equation*}
G_{i}\left( t_{j}+t_{0},t+t_{0}\right) >2^{\alpha _{i}}G_{b}\left(
t_{j}+t_{0}\right) ^{\alpha _{i}/\alpha _{b}}.
\end{equation*}%
Thus, for each $0\leq s\leq t_{j},$%
\begin{eqnarray*}
\int_{s+t_{0}}^{t+t_{0}}g_{i}\left( r\right) dr &\geq
&\int_{t_{j}+t_{0}}^{t+t_{0}}g_{i}\left( r\right) dr \\
&>&2^{\alpha _{i}}\left( \int_{0}^{t_{j}+t_{0}}g_{b}\left( r\right)
dr\right) ^{\alpha _{i}/\alpha _{b}}\geq 2^{\alpha _{i}}\left(
\int_{0}^{s+t_{0}}g_{b}\left( r\right) dr\right) ^{\alpha _{i}/\alpha _{b}},
\end{eqnarray*}%
hence$\ $%
\begin{equation*}
\tau _{i}=\frac{G_{i}\left( s+t_{0},t+t_{0}\right) ^{1/\alpha _{i}}}{%
G_{b}\left( s+t_{0}\right) ^{1/\alpha _{b}}}\geq 2\text{.}
\end{equation*}%
On the other hand, (\ref{cpsm1}) implies%
\begin{equation*}
p_{b}\left( G_{b}\left( s+t_{0}\right) ,0\right) \leq p_{b}\left(
G_{b}\left( t_{0}\right) ,0\right) =G_{b}\left( t_{0}\right) ^{-d/\alpha
_{b}}p_{b}\left( 1,0\right) \leq 1.
\end{equation*}%
Using (i) and (iii) of Lemma \ref{pden} we get 
\begin{eqnarray*}
p_{b}\left( G_{i}\left( s+t_{0},t+t_{0}\right) ^{\alpha _{b}/\alpha
_{i}},y-x\right)  &=&\tau _{i}^{-d}p_{b}\left( G_{b}\left( s+t_{0}\right) ,%
\tfrac{1}{\tau _{i}}(y-x)\right)  \\
&\geq &\tau _{i}^{-d}p_{b}\left( G_{b}\left( s+t_{0}\right) ,x\right)
p_{b}\left( G_{b}\left( s+t_{0}\right) ,y\right) .
\end{eqnarray*}%
From (\ref{umtei}), (iv) of Lemma \ref{pden} and Jensen's inequality we
deduce that%
\begin{eqnarray*}
u_{i}(t+t_{0},x) &\geq &c_{i}\int\nolimits_{0}^{t_{j}}h_{i}\left(
s+t_{0}\right)  \\
&&\times \int\nolimits_{\mathbb{R}^{d}}p_{b}\left(
G_{i}(s+t_{0},t+t_{0})^{\alpha _{b}/\alpha _{i}},y-x\right) u_{j}\left(
s+t_{0},y\right) ^{\beta _{i}}dyds \\
&\geq &c_{i}\int\nolimits_{0}^{t_{j}}\tau _{i}^{-d}h_{i}\left(
s+t_{0}\right) p_{b}\left( G_{b}\left( s+t_{0}\right) ,x\right) \overline{u}%
_{j}\left( s+t_{0}\right) ^{\beta _{i}}ds.
\end{eqnarray*}%
Then $u_{i}\left( t+t_{0},x\right) =\infty $. The definition (\ref{dubar})
of $\overline{u}_{i}$\ implies that $\overline{u}_{i}$ blows up in finite
time, and working as before we conclude that $u_{j}$ also blows up in finite
time.\hfill 
\end{proof}

In what follows by $c$ we mean a positive constant that may change from
place to place.

The following result is interesting in itself.

\begin{proposition}
\label{explocombi}Let $v_{i},f_{i}:[t_{0},\infty )\rightarrow \mathbb{R}$ be
continuous functions such that%
\begin{equation*}
v_{i}(t)\geq k+k\int\nolimits_{t_{0}}^{t}f_{i}(s)v_{j}\left( s\right)
^{\beta _{i}}ds,\ \ t\geq t_{0},
\end{equation*}%
where $k>0$ is a constant. Then $v_{i}$ blow up in finite time if 
\begin{equation*}
\int\nolimits_{t_{0}}^{\infty }\left( f_{i}(s)^{1/(\beta
_{i}+1)}f_{j}(s)^{1/(\beta _{j}+1)}\right) ^{\left( \beta _{i}+1\right)
(\beta _{j}+1)/(\beta _{i}+\beta _{j}+2)}ds=\infty .
\end{equation*}
\end{proposition}

\begin{proof}
Consider the system%
\begin{equation}
z_{i}(t)=\frac{k}{2}+k\int\nolimits_{t_{0}}^{t}f_{i}(s)z_{j}\left( s\right)
^{\beta _{i}}ds,\ \ t\geq t_{0}.  \label{ecauxcomexplo}
\end{equation}%
Let $N_{i}=\{t>t_{0}:z_{i}(s)<v_{i}(s),\ \ \forall s\in \lbrack 0,t]\}$. It
is clear that $N_{i}\neq \varnothing $. Let $e_{i}=\sup N_{i}$. Without loss
of generality suppose that $e_{i}\geq e_{j}$. If $e_{i}<\infty $, then the
continuity of $v_{j}-z_{j}$, yields%
\begin{equation*}
0=(v_{j}-z_{j})(e_{j})\geq \frac{k}{2}+k\int\nolimits_{t_{0}}^{e_{j}}f_{j}(s)%
\left[ v_{i}\left( s\right) ^{\beta _{j}}-z_{i}\left( s\right) ^{\beta _{j}}%
\right] ds\geq \frac{k}{2}.
\end{equation*}%
Therefore $z_{i}(t)\leq v_{i}(t),$ for each $t\geq t_{0}$.

Define%
\begin{equation}
Z(t)=\log z_{i}(t)z_{j}(t)\text{, \ }t\geq t_{0}.  \label{dprosol}
\end{equation}%
Then, by (\ref{ecauxcomexplo})%
\begin{eqnarray*}
Z^{\prime }(t) &=&\frac{f_{i}(t)z_{j}(t)^{\beta _{i}}}{z_{i}(t)}+\frac{%
f_{i}(t)z_{i}(t)^{\beta _{j}}}{z_{j}(t)} \\
&=&\frac{\left( f_{i}(t)^{1/(\beta _{i}+1)}z_{j}(t)\right) ^{\beta
_{i}+1}+\left( f_{j}(t)^{1/(\beta _{j}+1)}z_{i}(t)\right) ^{\beta _{j}+1}}{%
z_{i}(t)z_{j}(t)}.
\end{eqnarray*}%
From Proposition 1 (p.259) of \cite{Q-H} we see that for each $x,y>0$,%
\begin{equation*}
y^{\beta _{i}+1}+x^{\beta _{j}+1}\geq c(xy)^{\left( \beta _{i}+1\right)
(\beta _{j}+1)/(\beta _{i}+\beta _{j}+2)}.
\end{equation*}%
Using this and (\ref{dprosol}) we obtain%
\begin{eqnarray*}
Z^{\prime }(t) &\geq &c\left( f_{i}(t)^{1/(\beta _{i}+1)}f_{j}(t)^{1/(\beta
_{j}+1)}\right) ^{\left( \beta _{i}+1\right) (\beta _{j}+1)/(\beta
_{i}+\beta _{j}+2)} \\
&&\times \left( z_{i}(t)z_{j}(t)\right) ^{\left( \beta _{i}\beta
_{i}-1\right) /(\beta _{i}+\beta _{j}+2)} \\
&=&cF(t)\exp \left( \tfrac{\beta _{i}\beta _{i}-1}{\beta _{i}+\beta _{j}+2}%
Z(t)\right) ,
\end{eqnarray*}%
where $F$ is like (\ref{dfexpl}). Consider the equation%
\begin{equation*}
H^{\prime }(t)=cF(t)\exp \left( cH(t)\right) ,\ \ t>t_{0},\ \ H(t_{0})=2\log 
\tfrac{k}{2}.
\end{equation*}%
whose solution is%
\begin{equation*}
H(t)=\log \left( e^{-cH(t_{0})}-c^{2}\int_{t_{0}}^{t}F(s)ds\right) ^{-1/c}.
\end{equation*}%
Since $H\leq Z$ then the result follows from (\ref{cenfexplo}).\hfill 
\end{proof}

\section{Blow up results}

We begin with the:

\begin{proof}[Proof of Theorem \protect\ref{TeoPr}]
From (\ref{umtei}) and (\ref{reldensi})%
\begin{eqnarray*}
\lefteqn{u_{i}(t+t_{0},x)\geq \int\nolimits_{\mathbb{R}^{d}}c_{i}p_{b}%
\left( G_{i}\left( t_{0},t+t_{0}\right) ^{\alpha _{b}/\alpha
_{i}},y-x\right) u_{i}(t_{0},y)dy} \\
&&+\int\nolimits_{0}^{t}h_{i}\left( s+t_{0}\right) \int\nolimits_{\mathbb{R}%
^{d}}c_{i}p_{b}\left( G_{i}\left( s+t_{0},t+t_{0}\right) ^{\alpha
_{b}/\alpha _{i}},y-x\right) u_{j}^{\beta _{i}}\left( s+t_{0},y\right) dyds.
\end{eqnarray*}%
Multiplying by $p_{b}\left( G_{b}\left( t+t_{0}\right) ,x\right) $ and
integrating with respect to $x$ we get%
\begin{eqnarray*}
\lefteqn{\bar{u}_{i}\left( t+t_{0}\right) \geq c_{i}\int\nolimits_{\mathbb{R%
}^{d}}p_{b}\left( G_{i}\left( t_{0},t+t_{0}\right) ^{\alpha _{b}/\alpha
_{i}}+G_{b}\left( t+t_{0}\right) ,y\right) u_{i}(t_{0},y)dy} \\
&&+c_{i}\int\nolimits_{0}^{t}h_{i}\left( s+t_{0}\right) \int\nolimits_{%
\mathbb{R}^{d}}p_{b}\left( G_{i}\left( s+t_{0},t+t_{0}\right) ^{\alpha
_{b}/\alpha _{i}}+G_{b}\left( t+t_{0}\right) ,y\right)  \\
&&\times u_{j}^{\beta _{i}}\left( s+t_{0},y\right) dyds.
\end{eqnarray*}%
The property (ii) of Lemma \ref{pden} and Jensen's inequality, rendering%
\begin{eqnarray*}
\lefteqn{\bar{u}_{i}\left( t+t_{0}\right) \geq c_{i}\int\nolimits_{\mathbb{R%
}^{d}}p_{b}\left( G_{i}(t_{0},t+t_{0})^{\alpha _{b}/\alpha _{i}}+G_{b}\left(
t+t_{0}\right) ,y\right) u_{i}(t_{0},y)dy} \\
&&+c_{i}\int\nolimits_{0}^{t}\left( \frac{G_{b}\left( s+t_{0}\right) }{%
G_{i}\left( s+t_{0},t+t_{0}\right) ^{\alpha _{b}/\alpha _{i}}+G_{b}\left(
t+t_{0}\right) }\right) ^{d/\alpha _{b}} \\
&&\times h_{i}\left( s+t_{0}\right) \left( \bar{u}_{j}\left( s+t_{0}\right)
\right) ^{\beta _{i}}ds.
\end{eqnarray*}%
Moreover, (\ref{estuini}) and that $G_{i}\left( s,\cdot \right) $ is
increasing implies%
\begin{eqnarray*}
\lefteqn{\bar{u}_{i}\left( t+t_{0}\right) \geq c_{i}c_{i}(t_{0})p_{b}\left(
1,0\right) \left( 2G_{i}(t+t_{0})^{\alpha _{b}/\alpha _{i}}+2G_{b}\left(
t+t_{0}\right) \right) ^{-d/\alpha _{b}}} \\
&&+c_{i}\int\nolimits_{0}^{t}h_{i}\left( s+t_{0}\right) \left( \frac{%
G_{b}\left( s+t_{0}\right) }{2G_{i}\left( t+t_{0}\right) ^{\alpha
_{b}/\alpha _{i}}+2G_{b}\left( t+t_{0}\right) }\right) ^{d/\alpha
_{b}}\left( \bar{u}_{j}\left( s+t_{0}\right) \right) ^{\beta _{i}}ds.
\end{eqnarray*}%
Let 
\begin{equation*}
v_{i}(t+t_{0})=\bar{u}_{i}\left( t+t_{0}\right) (G_{i}\left( t+t_{0}\right)
^{\alpha _{b}/\alpha _{i}}+G_{b}\left( t+t_{0}\right) )^{d/\alpha _{b}},
\end{equation*}%
then
\begin{equation*}
v_{i}(t+t_{0})\geq c+c\int\nolimits_{0}^{t}f_{i}(s+t_{0})v_{j}\left(
s+t_{0}\right) ^{\beta _{i}}ds,
\end{equation*}%
where $f_{i}$ is defined in (\ref{dfi}). The result follows from Proposition %
\ref{explocombi} and Lemma \ref{caresp}.\hfill 
\end{proof}

\bigskip

\begin{proof}[Proof of Corollary \protect\ref{SegRe}]
Let 
\begin{equation*}
f_{i}(t)=\frac{t^{\sigma _{i}+d\rho _{b}/\alpha _{b}}}{(t^{\rho _{j}\alpha
_{b}/\alpha _{j}}+t^{\rho _{b}})^{d\beta _{i}/\alpha _{b}}},
\end{equation*}%
then%
\begin{equation*}
F(t)=\frac{t^{\theta _{1}}}{(t^{\theta _{2}}+t^{\theta _{3}})^{\theta
_{4}}(t^{\theta _{5}}+t^{\theta _{3}})^{\theta _{6}}}
\end{equation*}%
where%
\begin{eqnarray*}
\theta _{1} &=&\frac{d\rho _{b}}{\alpha _{b}}+\frac{\sigma _{i}(1+\beta
_{j})+\sigma _{j}(1+\beta _{i})}{2+\beta _{i}+\beta _{j}}, \\
\theta _{2} &=&\frac{\rho _{j}\alpha _{b}}{\alpha _{j}},\ \ \theta _{3}=\rho
_{b},\ \ \theta _{4}=\frac{d\beta _{i}(\beta _{j}+1)}{\alpha _{b}(2+\beta
_{i}+\beta _{j})}, \\
\theta _{5} &=&\frac{\rho _{i}\alpha _{b}}{\alpha _{i}},\ \ \theta _{6}=%
\frac{d\beta _{j}(\beta _{i}+1)}{\alpha _{b}(2+\beta _{i}+\beta _{j})}.
\end{eqnarray*}%
Using the elementary inequality%
\begin{equation*}
(t^{\theta _{2}}+t^{\theta _{3}})^{\theta _{4}}(t^{\theta _{5}}+t^{\theta
_{3}})^{\theta _{6}}\leq (2t^{\max \{\theta _{2},\theta _{3}\}})^{\theta
_{4}}(2t^{\max \{\theta _{5},\theta _{3}\}})^{\theta _{6}},\ \ t>1,
\end{equation*}%
the result follows.\hfill 
\end{proof}

\begin{flushleft}
\textsc{Jos\'{e} Villa Morales} \newline
Universidad Aut\'{o}noma de Aguascalientes\newline
Departamento de Matem\'{a}ticas y F\'{\i}sica\newline
Aguascalientes, Aguascalientes, M\'{e}xico.\newline
\texttt{jvilla@correo.uaa.mx}
\end{flushleft}

\end{document}